\documentclass[10pt]{amsart}

\usepackage[margin=1in]{geometry}
\usepackage{amsmath,amssymb,amsthm,mathtools}
\usepackage{tikz-cd}
\usepackage{microtype}
\usepackage[colorlinks=true,hyperindex,linkcolor=magenta,pagebackref,citecolor=cyan]{hyperref}

\theoremstyle{plain}
\newtheorem{theorem}{Theorem}[section]
\newtheorem{lemma}[theorem]{Lemma}
\newtheorem{proposition}[theorem]{Proposition}
\newtheorem{corollary}[theorem]{Corollary}
\theoremstyle{definition}

\newtheorem{question}[theorem]{Question}
\newtheorem{remark}[theorem]{Remark}

\title{Local cohomology modules with nonclosed support}
\author{Bhargav Bhatt}
\address{\begin{tabular}[t]{@{}l@{}} School of Mathematics, Institute for Advanced Study\\ Department of Mathematics, Princeton University \end{tabular}}
\email{bhargav.bhatt@gmail.com}
\author{Linquan Ma}
\address{Department of Mathematics, Purdue University, West Lafayette, IN 47907, USA}
\email{ma326@purdue.edu}

\begin{document}

\begin{abstract}
We construct noetherian rings admitting local cohomology modules with nonclosed support, or equivalently
with infinitely many minimal primes; this answers a question of Huneke--Lyubeznik. 
\end{abstract}

\maketitle

\section{Introduction}

For an ideal $I$ in a commutative ring $R$, Grothendieck's 1961 seminar
\cite{Har67} introduced local cohomology $H^*_I(-)$ of $R$-modules. The
construction provided a non-graded cohomological counterpart of Serre's
dictionary \cite{Ser55} relating coherent sheaves on projective varieties to
graded modules; it was developed further in SGA2 \cite{Gro68}, where it was
applied to Lefschetz theorems for Picard groups and fundamental groups. Since then,
local cohomology has become a central tool in the study of cohomological dimension
and vanishing theorems, (e.g., from the early work \cite{Har68,Ogus73} to the recent paper \cite{MustataPopaLocalCoh}). 

The $R$-modules $H^*_I(R)$ are seldom finitely generated, even in nice
situations. Nevertheless, beginning with Grothendieck's work, much effort has
been expended on finding ``hidden finiteness'' properties of these modules,
extending Serre's finiteness theorems for coherent cohomology of projective
varieties. Some of Grothendieck's early conjectures in this direction were
disproved by Hartshorne \cite{Har70}. On the positive side, Faltings
\cite{Fal78,Fal81} proved the existence of a nonzero annihilator ideal for
local cohomology modules in certain situations. In \cite{Hun92}, Huneke asked
whether local cohomology modules always have finitely many associated prime
ideals. The answer is affirmative for regular local rings in positive
characteristic \cite{HS93}, characteristic zero \cite{Lyu93}, and
unramified\footnote{Very recently, the second author has shown that the
statement fails for ramified regular local rings \cite{Ma26}.} mixed
characteristic \cite{Lyu00}, and also for smooth $\mathbb Z$-algebras
\cite{BBLSZ14}. The proofs in these regular settings are ``topological'':
they rely crucially on a $\mathcal D$-module or $F$-module structure.

Outside the regular case, examples of Singh, Katzman, and Singh--Swanson showed
that the set of associated primes may be infinite \cite{Sin00,Kat02,SS04}.
These examples still left open the following extremely natural question,
usually attributed to Huneke--Lyubeznik:

\begin{question}\label{q:huneke-lyubeznik}
Fix a noetherian ring $R$, ideal $I \subset R$, and integer $i$. Is $\operatorname{Supp} H^i_I(R)$ Zariski closed? Equivalently, since
$\operatorname{Supp} H^i_I(R)$ is closed under specialization, does
$\operatorname{Supp} H^i_I(R)$ have finitely many minimal primes?
\end{question}

Question~\ref{q:huneke-lyubeznik} was studied in depth in \cite{HKM09}; in
particular, they gave an affirmative answer when $\dim(R)\leq 4$ and $R$ is
local, or when $I$ has cohomological dimension at most $2$. Other positive
cases, mostly in characteristic $p$, include hypersurface and finite
$F$-representation type settings \cite{HNB17,KZ17,DaoQuy,QG24}, and more
recently certain complete intersections \cite{GZ25}. It is tempting to
speculate that a sufficiently robust $\mathcal D$-module/$F$-module theory for
singular varieties might extend the argument in the regular case to the
general case. This note shows that the simplest form of this expectation
fails as Question~\ref{q:huneke-lyubeznik} has a negative answer, even for normal domains:

\begin{theorem}[Theorems~\ref{thm:normal-five-quotient}
and~\ref{thm:normal-five-blowups-charp}, and
Corollary~\ref{cor:normal-local-six}]
\phantomsection
\label{main-theorem}
\begin{enumerate}
\item Let $k=\mathbb Z$ or let $k$ be a field. There is a finitely generated
flat normal $k$-domain $R$ of
dimension $5$ and a codimension $1$ prime ideal $I\subset R$ such that
$H^2_I(R)$ has infinitely many minimal primes.
\item Let $k$ be a field. There is a normal local domain $R$ of dimension
$6$, essentially of finite type over $k$, and a codimension $1$ prime ideal
$I\subset R$ such that $H^2_I(R)$ has infinitely many minimal primes.
\end{enumerate}
\end{theorem}

The infiniteness in (1) comes, ultimately, from combining two geometric phenomena: there exist one-parameter families of degree-zero line bundles on
abelian varieties whose generic member is nontorsion but which becomes torsion at infinitely many specializations; and, for a degree-zero line bundle $L$ on an abelian variety $A$, one has $H^1(A,L)\neq0$ if and only if $L$ is trivial. In fact, using this idea, we give two constructions of such an example: a simpler one in Theorem~\ref{thm:normal-five-quotient} that works when $k$ is not a field of characteristic $2$, and a slightly more involved one in Theorem~\ref{thm:normal-five-blowups-charp} that works uniformly for all $k$ of positive characteristic. Part (2) is obtained from part (1) by passage to affine cones. Separately, we also give a completely explicit arithmetic example, at the expense of allowing nilpotents:

\begin{theorem}[Proposition~\ref{prop:explicit-p-torsion}]
\label{thm:intro-p-torsion}
Consider the ring 
\[ R :=\frac{\mathbb Z[x_0,x_1,x_2,u_0,u_1,u_2]}
{(x_0u_1-x_1u_0,\ x_0u_2-x_2u_0,\ x_1u_2-x_2u_1)^2}.\]
This $5$-dimensional non-reduced ring admits a codimension $1$ prime ideal $J\subset R$ such that
$H^2_J(R)$ has infinitely many minimal primes.
\end{theorem}

In this example, the infiniteness arises from the fact that
$H^2_J(R)$ contains $p$-torsion for every prime number $p$, while it vanishes
rationally. In fact, the construction naturally gives
\[
H^2_J(R)\simeq H^1_{\text{dR}}(\mathbb{A}^1_{\mathbb{Z}}) \simeq  \bigoplus_{r\geq 1}\mathbb Z/r\mathbb Z
\]
as an abelian group.

\begin{remark}[A comment on dimension]
\label{rmk:dimintro}
The counterexamples in Theorem~\ref{main-theorem} (1) and 
Theorem~\ref{thm:intro-p-torsion} have dimension $5$. On the other hand, if $R$ is a noetherian ring of
dimension at most $3$, then $H^i_I(M)$ has only finitely many minimal primes
for every ideal $I\subset R$, every finitely generated $R$-module $M$, and
every $i$. The cases $i=0,1$ are elementary
\cite[Proposition~1.1(c)]{Mar01}, and the remaining cases follow from
\cite[Proposition~2.3 and Corollary~2.4]{Mar01} and
\cite[Corollary~3.3]{HKM09}. Thus dimension $4$ is the only unresolved case
without further hypotheses. Since finiteness is known for local rings of
dimension at most $4$ \cite{HKM09}, any $4$-dimensional counterexample to
Question~\ref{q:huneke-lyubeznik} would necessarily be nonlocal; we were unable to find one.

In the local case, Theorem~\ref{main-theorem} (2) gives a $6$-dimensional example, so the main remaining question is to find a $5$-dimensional local example.
\end{remark}

Besides the precise question in Remark~\ref{rmk:dimintro}, it would be interesting know if local cohomology modules of noetherian rings enjoy some yet-to-be-articulated finiteness property. 

\subsection*{Acknowledgments}
We are grateful to Johan de Jong, Mel Hochster, Craig Huneke, Anurag Singh,
and Ilya Smirnov for discussions. B.B. was partially supported by grants
from the Packard Foundation and the Simons Foundation (MPS-SIM-00622511,
MPS-PERF-00001529-02). L.M. was partially supported by NSF grants DMS-2302430
and DMS-2424441.

\subsection*{AI disclosure}
The examples in this paper arose through interactions with ChatGPT (5.5 and 5.6). The LLM's first ``solution'' of Question~\ref{q:huneke-lyubeznik} was flawed, but had the promising idea of using torsion points on elliptic curves; when encouraged to pursue that idea further, the model produced a $6$-dimensional counterexample. The authors modified and simplified this construction to obtain the $5$-dimensional normal example in Theorem~\ref{thm:normal-five-quotient}. In a separate interaction, the model found the example in Theorem~\ref{thm:intro-p-torsion} when asked for a local cohomology module that vanishes rationally but has $p$-torsion for infinitely many primes $p$. The authors wrote the paper (with assistance from Codex) and are responsible for its correctness.

\subsection*{Conventions}
If $X$ is a scheme and $\mathcal E$ is a vector bundle on $X$, we use the
Grothendieck conventions
\[
\mathbb P_X(\mathcal E)
:=\operatorname{Proj}_X\!\left(\operatorname{Sym}_{\mathcal O_X}\mathcal E\right),
\qquad
\mathbb V_X(\mathcal E)
:=\operatorname{Spec}_X\!\left(\operatorname{Sym}_{\mathcal O_X}\mathcal E^\vee\right).
\]
Thus $\mathbb P_X(\mathcal E)$ parametrizes invertible quotients of
$\mathcal E$, while $\mathbb V_X(\mathcal E)$ is the total space of
$\mathcal E$. For a line bundle $\mathcal{L}$ on $X$, writing $\pi:\mathbb{P}_X(\mathcal{O}_X \oplus \mathcal{L}^{-1}) \to X$ for the projection, the open subset of $\mathbb{P}_X(\mathcal{O}_X \oplus \mathcal{L}^{-1})$ where the second component of the tautological quotient map $\pi^*(\mathcal{O}_X \oplus \mathcal{L}^{-1}) \to \mathcal{O}_\pi(1)$ induces an isomorphism $\pi^* \mathcal{L}^{-1} \simeq \mathcal{O}_{\pi}(1)$ is identified $\mathbb{V}_X(\mathcal{L}^{-1})$; this embedding will be used in the sequel.

\section{A normal \texorpdfstring{$5$}{5}-dimensional example via a quotient}

In this section, we record an example proving Theorem~\ref{main-theorem} (1) away from characteristic $2$. It relies on the following:

\begin{proposition}[Infinite order points with many torsion specializations]
\label{prop:abelian-surface-input}
Let $k=\mathbb{Z}$ or let $k$ be a field. There is a finite type affine 
$k$-scheme $S$ with $\mathcal{O}(S)$ a PID, an elliptic curve
$E\to S$, and $P\in E(S)$ such that:
\begin{enumerate}
\item $P_\eta$ has infinite order in $E_\eta(\kappa(\eta))$, where $\eta \in S$ is the generic point.
\item there are infinitely many closed points $s\in S$ such that
$P_s$ has finite order in $E_s(\kappa(s))$.
\end{enumerate}
In fact, when $k=\mathbb{Z}$, we may take
$S=\operatorname{Spec}(\mathbb{Z}[1/N])$ for $N\geq 1$ sufficiently divisible;
when $k$ is a field, we may take $S$ to be a nonempty open subset of
$\mathbb{A}^1_k$.
\end{proposition}

In our eventual applications, it suffices to know that $\mathcal{O}(S)$ is a Dedekind domain. 

\begin{proof}
For $k=\mathbb{Z}$, take any elliptic curve $E_{\mathbb{Q}}/\mathbb{Q}$ with a
$\mathbb{Q}$-rational point $P$ of infinite order, for instance the curve
$37a1$ in \cite{Cremona}. After increasing $N$, the curve and $P$ spread out
over $S=\operatorname{Spec}(\mathbb{Z}[1/N])$. Then $\mathcal{O}(S)$ is a PID,
$P_\eta$ has infinite order, and every $P_s$ has finite order because each
closed fibre $E_s$ is over a finite field, so $E_s(\kappa(s))$ is a finite group. 

Now let $k$ be a field. We use a standard construction involving the identity map and quadratic twists; see \cite[Sections~8.1--8.3]{MoretBailly05} (where it
is stated in characteristic different from $2$). Choose an elliptic curve
$E_0$ over $k$. Let $\iota=[-1]$ and put $T=E_0-E_0[2]$, so the involution $\iota$ acts freely on
$T$. Then $T \to S := T/\iota$ is a finite \'etale degree $2$ map. Moreover, $S$
is a nonempty affine open subset of $E_0/\iota\simeq \mathbb{P}^1_k$ that
misses the image of $0$, so $\mathcal{O}(S)$ is a localization of $k[t]$ and
hence a PID. Projection to the second component descends to a map
\[
E:=(E_0\times T)/\iota \longrightarrow S,
\]
which is an elliptic curve over $S$ (with zero section descending from
$0\times T$). The diagonal map $\Delta_T:T \to E_0 \times T$ is $\iota$-equivariant, so 
it descends to a section $P:S\to E$. We claim that $(E \to S,P)$ does the job. As both claims can be checked after pulling back along finite covers, it suffices to show the analog of (1) and (2) for $(E \to S, P)$ replaced by $(E_0 \times T, \Delta_T)$. If (1) failed, then $\Delta_T$ would have image in $E_0[N] \times T$ for some integer $N \geq 1$, which is clearly false. And (2) clearly holds: simply run through the torsion points of $T = E_0 - E_0[2]$.
\end{proof}

The following standard lemma collects the facts about finite quotients that
we will use.

\begin{lemma}[Quotients of projective schemes]
\label{lem:finite-group-quotient}
Let $S=\operatorname{Spec}(R)$ be noetherian, let $G$ be a finite group of
order $N$ with $N\in R^\times$, and let $Y$ be a flat projective normal
integral $S$-scheme equipped with an $S$-action of $G$. Then the quotient $X := Y/G$ in schemes exists, the quotient map $q:Y \to X$ is finite, and $X$ is a flat projective normal integral $S$-scheme. More precisely, if $\mathcal A$ is an $S$-ample line bundle on $Y$, then some power of the naturally $G$-linearized line bundle $\bigotimes_{g\in G}g^*\mathcal A$ descends to an $S$-ample line bundle $\mathcal H$ on $X$. The corresponding $\mathbb{N}$-graded section ring $\bigoplus_{n\geq0}H^0(X,\mathcal H^{\otimes n}) = \left(\bigoplus_{n\geq0}H^0(Y,q^*\mathcal H^{\otimes n})\right)^G$ is a finitely generated normal domain over $R$ that,  after replacing $\mathcal H$ by a power, is generated in degree $1$.
\end{lemma}

\begin{proof}
This material is standard, so we only give references. The existence of the quotient follows from \cite[Theorem~4.4, Proposition~4.7, and Remark~4.9]{Rydh13}. The desired descended bundle $\mathcal{H}$ is provided by applying the criterion in \cite[Theorem~10.3]{Alper13}, and its ampleness follows from \cite[\href{https://stacks.math.columbia.edu/tag/0B5V}{Tag~0B5V}]{stacks-project}. The stated properties of the section ring are standard.
\end{proof}

We shall also repeatedly use the standard translation between cohomology on a quasi-projective variety and local cohomology of associated homogeneous co-ordinate rings:

\begin{lemma}[Local cohomology of graded rings via geometry]
\label{lem:cone-cohomology}
Fix a base ring $k$. Let $X/k$ be a projective $k$-scheme equipped with an ample line bundle $\mathcal{O}_X(1)$ and a closed subscheme $Z \subset X$, and put $U:=X-Z$. Let $R = \oplus_{n \geq 0} H^0(X,\mathcal{O}_X(n))$ be the homogeneous co-ordinate ring of $X$, and let $J \subset R$ be the homogeneous ideal of sections vanishing on $Z$. Then there is a fibre sequence
\[ R\Gamma_J(R) \to R \to \bigoplus_{n \in \mathbb{Z}} R\Gamma(U, \mathcal{O}_U(1)^{\otimes n}) \]
of graded objects in $D(R)$. In particular, for $i \geq 2$, we have an isomorphism 
\[ H^i_J(R) \simeq \bigoplus_{n \in \mathbb{Z}} H^{i-1}(U, \mathcal{O}_U(1)^{\otimes n})\]
of graded $R$-modules.
\end{lemma}
\begin{proof}
  The cofibre of $R\Gamma_J(R) \to R$ is identified with $R\Gamma(\mathrm{Spec}(R)-V(J),\mathcal{O})$ by definition. Moreover, for a line bundle $\mathcal{L}$ on a scheme $Y$, the total space $\mathcal{L}^\times$ of the $\mathbf{G}_m$-torsor over $Y$ attached to $\mathcal{L}$ is an affine $Y$-scheme corresponding to the quasi-coherent sheaf of algebras $\bigoplus_{n \in \mathbb{Z}} \mathcal{L}^{\otimes n}$. Consequently, it is enough to identify $\mathrm{Spec}(R)-V(J)$ with $\mathcal{O}_U(1)^\times \simeq \mathcal{O}_X(1)^\times \times_X U$. For this, write  $R_+$ for the ideal of positive degree elements in $R$, and similarly for $(R/J)_+$.  Since $X=\mathrm{Proj}(R)$, the scheme $\mathrm{Spec}(R)-V(R_+)$ is $\mathcal{O}_X(1)^\times$. Moreover, since $R/J$ agrees in sufficiently large degrees with the homogeneous co-ordinate ring of $Z$, we have $Z=\mathrm{Proj}(R/J)$, so $\mathrm{Spec}(R/J) - V( (R/J)_+)$ is $\mathcal{O}_Z(1)^\times \simeq \mathcal{O}_X(1)^\times \times_X Z$. Putting these together shows that  $\mathrm{Spec}(R) - V(J)$ is $\mathcal{O}_U(1)^\times$, as wanted.
\end{proof}

The promised example is the following:

\begin{theorem}
\label{thm:normal-five-quotient}
Let $k=\mathbb Z$ or let $k$ be a field of characteristic different from
$2$. There is a finitely generated flat normal $k$-domain $T$ of dimension
$5$ and a codimension $1$ prime ideal $J\subset T$ such that $H^2_J(T)$ has
infinitely many minimal primes.
\end{theorem}

\begin{proof}
Choose $S$, $E/S$, and $P\in E(S)$ as in
Proposition~\ref{prop:abelian-surface-input}. Taking suitable self-products, we can find an abelian surface $A/S$, a rigidified\footnote{A rigidified line bundle on an abelian scheme $B \to T$ is a line bundle $M$ on $B$ together with a trivialization along the identity
section such that $M$ has degree $0$ on the geometric fibres $B_t$. Such a pair determines a commutative $T$-group structure on the total space $M^\times$ of the $\mathbf{G}_m$-torsor attached to $M$, compatible with the one on $B$, so $M^\times$ is an extension of $B$ by $\mathbb{G}_m$.} line bundle $\mathcal L\in\operatorname{Pic}^0(A/S)$ whose generic fibre has infinite order and whose restriction $\mathcal{L}_s \in \mathrm{Pic}^0(A_s)$ is torsion for infinitely many closed points $s \in S$. In the case
$k=\mathbb Z$, we may enlarge $S$ to assume $2$ is invertible on $S$. Form the diagram
\[
\begin{tikzcd}[column sep=large]
W:=\operatorname{Spec}_A\left(\displaystyle\bigoplus_{r\in\mathbb Z}
\mathcal L^{\otimes r}\right)
    \arrow[r,hook] \arrow[dr,"\pi"'] &
\overline W:=\mathbb P_A(\mathcal O_A\oplus\mathcal L^{-1})
    \arrow[d,"\overline\pi"] \\
& A.
\end{tikzcd}
\]
Write $D_0 \subset \overline{W}$
and $D_\infty \subset \overline{W}$ for the two evident sections over $\overline{\pi}$, and set
$B=D_0\sqcup D_\infty$. The rigidification of $\mathcal{L}$ makes $W$ into a commutative group scheme fitting into an
exact sequence
\[
1\longrightarrow\mathbb G_m\longrightarrow W\longrightarrow A
\longrightarrow 0.
\]
Let $\tau=[-1]_W$ be inversion in this group scheme; this lifts $[-1]_A$ and
restricts to inversion on $\mathbb G_m$, so it extends to $\overline W$,  where it exchanges $D_0$ and $D_\infty$.

Set $G=\langle\tau\rangle$, so $G$ acts compatibly on both $\overline{W}$ and $A$. Let $\mathcal N$ be a sufficiently
ample line bundle on $A$ such that
\[
\mathcal A:=
\mathcal O_{\overline W}(D_0+D_\infty)
\otimes\overline\pi^*\mathcal N
\]
is ample on $\overline W$. The scheme $\overline W$ is a flat projective
normal integral $S$-scheme, so Lemma~\ref{lem:finite-group-quotient} gives a
finite quotient
\[
q:\overline W\longrightarrow X:=\overline W/G
\]
with $X$ flat, projective, normal, and integral over $S$, and an ample line
bundle $\mathcal H$ on $X$ such that, for some $e>0$,
\begin{equation}\label{eq:quotient-descent}
q^*\mathcal H\simeq
(\mathcal A\otimes\tau^*\mathcal A)^{\otimes e}.
\end{equation}
We may also assume that the full section ring of $\mathcal H$ is generated
in degree $1$.

Put $D=B/G$ and $U=W/G=X-D$. Since $\tau$ exchanges the two components of $B$, it acts freely in a
neighbourhood of $B$ (in fact $A \simeq D$ via either component of $B$) and $D \to X$ is the inclusion of a Cartier divisor. These fit into the commutative diagram
\[
\begin{tikzcd}[column sep=large,row sep=large]
B := D_0 \sqcup D_\infty \arrow[r,hook] \arrow[d,"{\text{finite \'etale}}"'] &
\overline W \arrow[d,"{q\ \text{finite}}"'] &
W \arrow[l,hook'] \arrow[d] \\
D  \arrow[r,hook] & X & U \arrow[l,hook'] 
\end{tikzcd}
\]
with Cartesian squares. As $D_0 \cap W = D_\infty \cap W = \emptyset$, we have
\begin{equation}\label{eq:normal-quotient-polarization}
q^*\mathcal H|_W\simeq\pi^* (\mathcal N\otimes[-1]_A^*\mathcal N)^{\otimes e}.
\end{equation}
Set
\[
T:=\bigoplus_{n\geq0}H^0(X,\mathcal H^{\otimes n})
\]
and let $J\subset T$ be the homogeneous ideal cutting out $D$. The ring $T$
is finitely generated and normal by Lemma~\ref{lem:finite-group-quotient}.
(For $k=\mathbb Z$, the ring $T$ is torsion-free and hence $\mathbb{Z}$-flat.) Since $\dim(X)=4$, we have $\dim(T)=5$. Restriction to $D$ has kernel $J$
and target a domain, so $J$ is prime. By Serre vanishing, $T/J$ agrees in all
sufficiently large degrees with the section ring of
$(D,\mathcal H|_D)$, so $\dim(T/J)=4$ and $J$ has codimension $1$. Lemma~\ref{lem:cone-cohomology} gives
\begin{equation}\label{eq:normal-quotient-cone}
H^2_J(T)\simeq
\bigoplus_{m\in\mathbb Z}H^1(U,\mathcal H^{\otimes m}|_U).
\end{equation}
Since $2$ is invertible on $S$, taking $G$-invariants is exact, so the finite
quotient $W\to U$ (and \eqref{eq:normal-quotient-polarization}) gives
\[
R\Gamma(U,\mathcal H^{\otimes m}|_U)
\simeq R\Gamma(W,\pi^*\mathcal M^{\otimes m})^G,
\]
where $\mathcal{M} = (\mathcal{N} \otimes [-1]_A^* \mathcal{N})^{\otimes e}$. As $\pi:W\to A$ is affine,
\begin{equation}\label{eq:normal-quotient-cohomology}
H^1(W,\pi^*\mathcal M^{\otimes m})
\simeq
\bigoplus_{r\in\mathbb Z}
H^1(A,\mathcal L^{\otimes r}\otimes\mathcal M^{\otimes m}).
\end{equation}
If $m>0$, all these groups vanish: by cohomology and base change, it is enough to prove the analogous statement for the fibres over points of $S$, where it follows from standard facts on abelian varieties \cite[Vanishing theorem on page~150, Corollary on page~159]{MumfordAV}. If $m<0$, they again vanish by the same argument combined with Serre duality.

For $m=0$, the involution exchanges the summands indexed by $r$ and $-r$ in
\eqref{eq:normal-quotient-cohomology}.
On the $r=0$ summand $H^1(A,\mathcal O_A)$ it acts by $-1$. Since $2$ is
invertible, this summand has no invariants, while each pair indexed by
$\{r,-r\}$, with $r>0$, contributes one copy of
$H^1(A,\mathcal L^{\otimes r})$ to the invariants. As the $m\neq0$ summands of \eqref{eq:normal-quotient-cone} vanish, we learn that 
\begin{equation}\label{eq:normal-quotient-final}
H^2_J(T)\simeq
\bigoplus_{r>0}H^1(A,\mathcal L^{\otimes r}),
\end{equation}
is a module over $T/T_+ = T_0=H^0(X,\mathcal O_X)=H^0(\overline W,\mathcal O_{\overline W})^G
=\mathcal O(S)$; we regard its support as a subset of
$S=V(T_+)\subset\operatorname{Spec}(T)$. The generic point $\eta\in S$ does
not belong to this support: $\mathcal L_\eta^{\otimes r}$ is a
nontrivial degree-zero line bundle for every $r>0$, and hence has no
cohomology \cite[Corollary on page~159]{MumfordAV}.  Now let $s\in S$ be a closed point for which $\mathcal L_s$ is torsion, let
$\pi_s\in\mathcal O(S)$ be a uniformizer, and choose $r>0$ such that
$\mathcal L_s^{\otimes r}\simeq\mathcal O_{A_s}$. Since
$H^0(A,\mathcal L^{\otimes r})=0$ by generic nontriviality of $\mathcal{L}^{\otimes r}$, the long exact
sequence on cohomology for multiplication by $\pi_s$ gives
\[
H^1(A,\mathcal L^{\otimes r})[\pi_s]
\simeq H^0(A_s,\mathcal O_{A_s})\neq0.
\]
Thus every such $s$ is an associated point of $H^2_J(T)$. Since $S$ is
one-dimensional and its generic point does not belong to the support, each
of these closed points is minimal in the support. As there are infinitely
many $s\in S$ for which $\mathcal L_s$ is torsion, we obtain the desired infinite set of minimal associated points.
\end{proof}

\begin{remark}
The preceding construction also works when $k$ has characteristic $2$, but the proof must be modified as taking $G$-invariants is not exact. More precisely,  the contributions of $m \neq 0$ summands to \eqref{eq:normal-quotient-cone} vanish as before (via vanishing theorems), as do the contributions of the $r \neq 0$ summand to \eqref{eq:normal-quotient-cohomology} (as the action switching the $r$-th summand with the $(-r)$-th summand is a free action of $G$ on the summand indexed by the pair $\{r,-r\}$ for $r \neq 0$). For the $r=0$ summand, one must show that $H^1(A/G,\mathcal{O}_{A/G}) = 0$, which was checked in enough special cases in \cite{ShiodaKummer}. We omit further details as \S \ref{sec:normal-contract} gives another construction that works uniformly in all characteristics.
\end{remark}

\section{A normal $5$-dimensional  example via contractions}
\label{sec:normal-contract}

In this section, we show the following, which in particular proves Theorem~\ref{main-theorem} (1) uniformly in all characteristics, including $2$.

\begin{theorem}
\label{thm:normal-five-blowups-charp}
Let $k$ be a field of characteristic $p>0$. There is a finitely generated flat
normal $k$-domain $T$ of dimension $5$ and a codimension $1$ prime ideal
$J\subset T$ such that $H^2_J(T)$ has infinitely many minimal primes.
\end{theorem}

Our construction will be a variant of the one used to prove Theorem~\ref{thm:normal-five-quotient}: instead of using finite quotients to kill certain degree $0$ summands, we use a birational contraction to do so. The proof will occupy the rest of this section. 

\subsection{The blowup and its cohomology}\label{subsec:blowup-cohomology}

In this subsection, we construct the blowup used in our construction, and compute the cohomology of some relevant line bundles.

Choose $S$, an abelian
surface $p:A\to S$, and a line bundle
$\mathcal L\in\operatorname{Pic}^0(A/S)$ such that $\mathcal O(S)$ is a PID,
$\mathcal L_\eta$ has infinite order, and $\mathcal L_s$ is torsion for
infinitely many closed points $s\in S$; this can be done as in the opening of the proof of
Theorem~\ref{thm:normal-five-quotient}. Form the diagram
\[
\begin{tikzcd}[column sep=large]
W:=\operatorname{Spec}_A\left(\displaystyle\bigoplus_{r\geq 0}
\mathcal L^{\otimes r}\right)
    \arrow[r,hook] \arrow[dr,"\pi"'] &
\overline W:=\mathbb P_A(\mathcal O_A\oplus\mathcal L^{-1})
    \arrow[d,"\overline\pi"] \\
& A,
\end{tikzcd}
\]
Let $D_0\subset W\subset\overline W$ be the zero section and let
$D_\infty=\overline W-W$ be the section at infinity. In other words, $D_0$ corresponds to the second projection $\mathcal{O}_A \oplus \mathcal{L}^{-1} \to \mathcal{L}^{-1}$
while $D_\infty$ corresponds to the first projection. If one could contract $D_0$, then it would be rather straightforward to show that a suitable section ring for the contracted variety verified Theorem~\ref{thm:normal-five-blowups-charp}. But $D_0$ cannot be contracted as its normal bundle is numerically trivial. Instead, we shall blow it up along an ample divisor, and then contract it. Choose a sufficiently ample line bundle $\mathcal M$ on $A/S$ such that 
\[
\mathcal M=\mathcal O_A(C)
\]
for an effective relative Cartier divisor $C\subset A$, smooth over $S$.
Blow up $C\subset D_0\subset\overline W$:
\[
b:\widetilde{\overline W}:=\operatorname{Bl}_C(\overline W)
\longrightarrow\overline W.
\]
Write $E$ for the exceptional divisor and $D$ for the strict transform of
$D_0$. Since $C$ is Cartier in $D_0$, the restriction of $b$ identifies
$D$ with $D_0$, and hence with $A$. This construction is summarized in the following commutative diagram
\[
\begin{tikzcd}[column sep=large,row sep=normal]
C=\mathbb P_C(N_{C/D_0}^{\vee}) \arrow[rr] \arrow[d,hook] &&
E=\mathbb P_C(N_{C/\overline W}^{\vee}) \arrow[rr,"b"] \arrow[d,hook] &&
C \arrow[d,hook,"{\mathrm{ample}}"'] \\
A\simeq D \arrow[rr,hook] &&
b^{-1}(D_0)=D\cup E \arrow[rr,"b"] \arrow[d,hook] &&
D_0\simeq A \arrow[d,hook,"{\mathrm{divisor}}"'] \\
&& \widetilde W:=\operatorname{Bl}_C(W) \arrow[rr,"b"] \arrow[d,hook] &&
W=\mathbb V_A(\mathcal L^{-1}) \arrow[d,hook,"{\mathrm{open}}"'] \\
&& \widetilde{\overline W}:=\operatorname{Bl}_C(\overline W)
  \arrow[rr,"b"]
  \arrow[ddrr,bend right=15,"\widetilde q"'] &&
\overline W=\mathbb P_A(\mathcal O_A\oplus\mathcal L^{-1})
  \arrow[d,"\overline\pi"]
  \arrow[dd,bend left=45,"\overline q"] \\
&&&& A \arrow[d,"p"] \\
&&&& S.
\end{tikzcd}
\]
where all squares are pullback squares (with the squares relating columns 2 and 3 being derived pullback squares), the maps $\overline{q}$ and $\widetilde{q}$ are  defined to make the diagram commute, and the horizontal compositions in the top two rows are the natural isomorphisms. Any unspecified closed immersion out of $C$ or $D$ is the one coming from the above diagram. Our goal is to contract $A \simeq D \subset \widetilde{\overline{W}}$. For this, consider the line bundles
\[
\begin{aligned}
\mathcal N
&=\mathcal O_{\overline W}(2D_\infty)
  \otimes\overline\pi^*\mathcal M
&&\in\operatorname{Pic}(\overline W), \text{ and } \\
\mathcal Q
&=b^*\mathcal N\otimes
  \mathcal O_{\widetilde{\overline W}}(-E)
&&\in\operatorname{Pic}(\widetilde{\overline W}).
\end{aligned}
\]
We arrange the choice of $\mathcal{M}$ so $\mathcal{N}$ is ample. We shall eventually use $\mathcal{Q}$ to carry out the promised contraction. For future calculations, we will need the line bundle 
\begin{equation}\label{eq:blowup-G}
\mathcal G:=b^*\pi^*\mathcal M\otimes
\mathcal O_{\widetilde W}(-E)
\simeq\mathcal Q|_{\widetilde W} \in \mathrm{Pic}(\widetilde{W})
\end{equation}
as well as cohomology of all of its powers on $\widetilde{W}$, as recorded next:

\begin{lemma}[Cohomology of blowups]
\label{lem:blowup-tensor-power-cohomology}
For every $d\in\mathbb Z$,
\[
H^1(\widetilde W,\mathcal G^{\otimes d})\simeq
\begin{cases}
H^1(A,\mathcal O_A),&d>0,\\[2pt]
\displaystyle\bigoplus_{r\geq0}H^1(A,\mathcal L^{\otimes r}),&d=0,\\[6pt]
0,&d<0.
\end{cases}
\]
Moreover, $\mathcal G|_D\simeq\mathcal O_D$, and under the identification
$D\simeq A$, the restriction map
\[
H^1(\widetilde W,\mathcal G^{\otimes d})
\longrightarrow H^1(D,\mathcal O_D)=H^1(A,\mathcal O_A)
\]
is an isomorphism if $d>0$, projection onto the $r=0$ summand if $d=0$,
and the zero map if $d<0$.
\end{lemma}

\begin{proof}
The restriction of $b$ identifies $D$ with $D_0$, and $E|_D=C$. Hence
\[
\mathcal G|_D\simeq
\mathcal M\otimes\mathcal O_A(-C)\simeq\mathcal O_D.
\]
Let $\mathcal I_C\subset\mathcal O_W$ be the ideal of $C$. For $d>0$, the
standard blowup calculation gives
\[
b_*\mathcal G^{\otimes d}
=(\pi^*\mathcal M^{\otimes d}) \otimes \mathcal I_C^d,
\qquad
R^ib_*\mathcal G^{\otimes d}=0\quad(i>0).
\]
The action of $\mathbb G_m$ scaling the fibres of $W$ identifies
$\mathbb G_m$-equivariant quasi-coherent sheaves on $W$ with graded
quasi-coherent modules over the graded $\mathcal O_A$-algebra
\[
\pi_*\mathcal O_W=\bigoplus_{r\geq0}\mathcal L^{\otimes r}.
\]
The ideal $\mathcal I_C$ is $\mathbb G_m$-stable and corresponds to the
graded submodule
\[
\pi_*\mathcal I_C
=\mathcal M^{-1}\oplus\bigoplus_{r>0}\mathcal L^{\otimes r}
\subset
\mathcal O_A\oplus\bigoplus_{r>0}\mathcal L^{\otimes r} = \pi_* \mathcal{O}_W
\]
By keeping track of indices, one learns that the degree-$r$ piece of $\pi_*\mathcal I_C^d$ is $\mathcal L^{\otimes r}\otimes
\mathcal M^{-\max\{d-r,0\}}$. Tensoring by
$\mathcal M^{\otimes d}$ therefore gives
\[
\pi_*b_*\mathcal G^{\otimes d}
\simeq\bigoplus_{r\geq0}
\mathcal L^{\otimes r}\otimes
\mathcal M^{\otimes\min\{r,d\}}.
\]
For $r>0$, the corresponding summand is an ample line bundle tensored with
an element of $\operatorname{Pic}^0(A/S)$, so it has no higher cohomology.
The $r=0$ summand is $\mathcal O_A$. This proves the assertion for $d>0$;
restriction to $D$ is projection onto the $r=0$ summand, so it induces the
asserted isomorphism on $H^1$.

For $d=0$, the blowup does not change structure-sheaf cohomology, and hence
\[
H^1(\widetilde W,\mathcal O_{\widetilde W})
\simeq\bigoplus_{r\geq0}H^1(A,\mathcal L^{\otimes r}).
\]
Restriction to the zero section, and hence to its strict transform $D$, is
projection onto the $r=0$ summand.

Finally, for $d < 0$, we must prove a vanishing theorem. Write $d=-n$ with $n>0$. We have
$b_*\mathcal O_{\widetilde W}(nE)=\mathcal O_W$, so the projection formula
gives $b_*\mathcal G^{-n}\simeq\pi^*\mathcal M^{-n}$. Therefore
\[
H^1(W,b_*\mathcal G^{-n})
\simeq H^1(W,\pi^*\mathcal M^{-n})
\simeq\bigoplus_{r\geq0}
H^1(A,\mathcal M^{-n}\otimes\mathcal L^{\otimes r})=0
\]
by Serre duality and the vanishing theorem for ample line bundles on the
abelian surface. By the Leray spectral sequence for $b$ and the projection formula, it therefore remains to check that
\begin{equation}\label{eq:blowup-negative-R1-vanishing}
H^0\!\left(W,\pi^*\mathcal M^{-n}\otimes
R^1b_*\mathcal O_{\widetilde W}(nE)\right)=0.
\end{equation}
On the exceptional divisor, $b:E=\mathbb P_C(N_{C/W}^{\vee})\to C\subset W$
is the projection followed by the inclusion. We regard sheaves on $C$ as
sheaves on $W$ via extension by zero. Since
$\mathcal O_E(E)=\mathcal O_E(-1)$, pushing forward the sequence
\[ 0 \to \mathcal{O}_{\widetilde{W}} \to \mathcal{O}_{\widetilde{W}}(E) \to \mathcal{O}_E(E) = \mathcal{O}_E(-1) \to 0 \]
shows that
$R^1b_*\mathcal O_{\widetilde W}(E)=0$. For $n\geq2$, inductively pushing forward twists of the above sequence shows that
$R^1b_*\mathcal O_{\widetilde W}(nE)$ has a filtration with successive
quotients
\[
R^1b_*\mathcal O_E(-j),\qquad 2\leq j\leq n.
\]
Now
\begin{equation}\label{eq:blowup-exceptional-quotients}
\begin{aligned}
R^1b_*\mathcal O_E(-j)
&\simeq
\left(\operatorname{Sym}^{j-2}(N_{C/W}^{\vee})\right)^{\vee}
\otimes_{\mathcal O_C}\det(N_{C/W}) \\
&=\Gamma^{j-2}(N_{C/W})\otimes_{\mathcal O_C}\det(N_{C/W}),
\qquad 2\leq j\leq n,
\end{aligned}
\end{equation}
where the isomorphism is the rank-two case of the projective-bundle
cohomology formula \cite[Tag 01XX]{stacks-project} (equivalently, relative
duality), and
$\Gamma^m(\mathcal E):=(\operatorname{Sym}^m(\mathcal E^\vee))^\vee$
denotes the $m$th divided power. The normal sequence for $C \hookrightarrow A \hookrightarrow W$ is
\[
0\longrightarrow\mathcal M|_C
\longrightarrow N_{C/W}
\longrightarrow\mathcal L^{-1}|_C
\longrightarrow0.
\]
The divided-power filtration associated to this sequence shows that, after
tensoring the sheaves in \eqref{eq:blowup-exceptional-quotients} by
$\mathcal M^{-n}|_C$, they have filtrations with graded pieces
\[
\mathcal M^{\otimes(a+1-n)}|_C\otimes_{\mathcal O_C}
\mathcal L^{\otimes(a-j+1)}|_C,
\qquad 0\leq a\leq j-2.
\]
Since $j\leq n$ and $a \leq j-2$, every exponent $a+1-n$ is negative. As $\mathcal M|_C$ is
relatively ample and $\mathcal L|_C$ has degree zero on every fibre, these
line bundles have negative fibrewise degree, and hence have no global
sections, which gives \eqref{eq:blowup-negative-R1-vanishing}.
\end{proof}

\subsection{Semiampleness and the contraction}\label{subsec:semiample-contraction}

In this subsection, we prove the required semiampleness statement and use it to construct the contraction of the strict transform of the zero section. In the next lemma, for $s \in S$, use a subscript of $s$ to denote the fibre over $s$.

\begin{lemma}[Semiampleness of $\mathcal{Q}$]
\label{lem:blowup-semiample}
After choosing $C$ sufficiently ample, the line bundle $\mathcal Q$ is
$\widetilde q$-semiample. Moreover, for $s \in S$, the exceptional locus for $\mathcal{Q}_s$ is $D_s$.
\end{lemma}

\begin{proof}
Fix a point $s\in S$. By \cite[Theorem~1.1]{CT20}, it is enough to show that $\mathcal Q_s$ is
semiample. We shall do so using Keel's criterion \cite[Theorem~0.2]{Kee99}: we
shall prove that $\mathcal Q_s$ is nef and big, identify its exceptional locus
$\mathbb E(\mathcal Q_s)$ with $D_s$, and prove that $\mathcal Q_s$ is semiample on
$\mathbb E(\mathcal Q_s)$.

To get started, we record two observations about $\mathcal{Q}$. First, it is trivial on $D$:
\begin{equation}\label{eq:blowup-Q-on-D}
\mathcal Q|_D\simeq
\mathcal N|_{D_0}\otimes\mathcal O_D(-E)
\simeq\mathcal M\otimes\mathcal O_A(-C)
\simeq\mathcal O_D.
\end{equation}
Secondly, we give another description of $\mathcal{Q}$. Consider the line bundle 
\[ \mathcal N':=\mathcal N(-D_0)
=\mathcal O_{\overline W}(2D_\infty-D_0)
\otimes\overline\pi^*\mathcal M \in \mathrm{Pic}(\overline{W}).\] 
It has degree $1$ on the fibres of
$\overline\pi$, so, after making $C$ sufficiently ample, it is $\overline q$-ample. Since
$D=b^*D_0-E$, we have
\begin{equation}\label{eq:blowup-Q-Nprime-D}
\mathcal Q = b^* \mathcal{N} \otimes \mathcal{O}_{\widetilde{\overline{W}}}(-E) \simeq
b^*\mathcal N'\otimes\mathcal O_{\widetilde{\overline W}}(D).
\end{equation}

We can now get started with checking Keel's criteria. First, we show $\mathcal{Q}_s$ is nef and big. In fact, bigness is clear from \eqref{eq:blowup-Q-Nprime-D}. For nefness, let $\Gamma$ be an integral curve  in $\widetilde{\overline{W}}_s$. If $\Gamma\subset D_s$, then
\eqref{eq:blowup-Q-on-D} gives $\mathcal{Q}_s \cdot \Gamma=0$, so $\mathcal{Q}_s$ is nef on curves in $D$. Moreover, if $\Gamma\not\subset D_s$, then we claim that the number 
\[ \mathcal Q_s\cdot \Gamma = (b^* \mathcal{N}')_s \cdot \Gamma + \mathcal{O}_{\widetilde{\overline{W}}}(D)_s \cdot \Gamma\] 
is actually positive. Indeed, both terms are non-negative as the corresponding line bundles are effective on $\Gamma$ by the assumption $\Gamma \not\subset D_s$. The first summand is positive if $b_s(\Gamma)$ is a curve (since $\mathcal{N}'$ is ample), and the second summand is positive if $b_s(\Gamma)$ is a point (since then $\Gamma$ is a fibre of $E_s \to C_s$ and $D \cap E_s$ is a section of this map), so the sum is always positive.  This proves that $\mathcal{Q}_s$ is nef and in fact positive on curves not in $D_s$.

Next, we show that $\mathbb E(\mathcal Q_s)=D_s$.
We have just seen that the curves on which $\mathcal Q_s$ has degree zero lie in $D_s$. If $Y \subset \widetilde{\overline{W}}_s$ is an integral surface with $Y\not\subset D_s$, then \eqref{eq:blowup-Q-Nprime-D} shows that $\mathcal{Q}_s|_Y$ is big provided $b_s(Y)$ is a surface. It remains to consider $Y=E_s$, i.e., to show that $\mathcal{Q}_s|_{E_s}$ is big. Since $D_\infty \cap D_0 = \emptyset$, we have $\mathcal N_s|_{C_s}\simeq\mathcal M_s|_{C_s}$, and hence
\[
\mathcal Q_s|_{E_s}\simeq
b_s^*(\mathcal M_s|_{C_s})
\otimes\mathcal O_{\mathbb P_{C_s}(N_{C_s/\overline W_s}^{\vee})}(1)
\simeq\mathcal O_{\mathbb P_{C_s}(\mathcal V_s)}(1),
\qquad
\mathcal V_s:=N_{C_s/\overline W_s}^{\vee}\otimes\mathcal M_s|_{C_s}.
\]
The dual normal sequence for
$C_s\hookrightarrow D_s\hookrightarrow\overline W_s$ is
\[
0\longrightarrow
\mathcal L_s|_{C_s}
\longrightarrow N_{C_s/\overline W_s}^{\vee}
\longrightarrow\mathcal M_s^{-1}|_{C_s}
\longrightarrow0.
\]
Tensoring it with $\mathcal M_s|_{C_s}$ gives
\[
0\longrightarrow
(\mathcal L_s\otimes\mathcal M_s)|_{C_s}
\longrightarrow\mathcal V_s
\longrightarrow\mathcal O_{C_s}
\longrightarrow0.
\]
As the first term above is an ample line bundle, a simple analysis using the above sequence shows that $\mathcal V_s$ is a nef vector bundle, i.e., that $\mathcal O_{\mathbb P_{C_s}(\mathcal V_s)}(1)$ is nef. Moreover, by \cite[\href{https://stacks.math.columbia.edu/tag/02U0}{Tags~02U0} and \href{https://stacks.math.columbia.edu/tag/0ERU}{0ERU}]{stacks-project}, the self-intersection number $c_1(\mathcal O_{\mathbb P_{C_s}(\mathcal V_s)}(1))^2$ is given by $\deg(\mathcal V_s)=\deg(\mathcal M_s|_{C_s})>0$, so we get bigness of $\mathcal{Q}_s|_{E_s}$; alternately, one can estimate $H^0(C_s,\mathrm{Sym}^n(\mathcal{V}_s))$ directly using the above sequence to prove bigness.

It remains to show that $\mathcal{Q}|_{D_s}$ is semiample on $D_s$. But $\mathcal Q|_{D_s}$ is trivial by \eqref{eq:blowup-Q-on-D}, so Keel's criterion applies.
\end{proof}

By the previous lemma, we can choose $e>0$ such that $\mathcal Q^{\otimes e}$ is relatively globally generated. Its associated morphism $\varphi$ can be Stein factored as 
\[
\varphi:\widetilde{\overline W}\xrightarrow{\rho}X
\xrightarrow{g}
\mathbb P_S(\widetilde q_*\mathcal Q^{\otimes e}),
\]
so $X$ is normal, the line bundle $\mathcal H:=
g^*\mathcal O_{\mathbb P_S(\widetilde q_*\mathcal Q^{\otimes e})}(1)$
is ample, and $\rho^*\mathcal H\simeq\mathcal Q^{\otimes e}$. As $\mathbb{E}(\mathcal Q_s) = D_s$ for all $s \in S$ and $\mathcal{Q}|_D \simeq \mathcal{O}_D$, the proper birational map $\rho$ collapses $D$ to an $S$-point $i:S \to X$, does not send any other point of $\widetilde{\overline{W}}$ to $i(S)$, and the induced map $\widetilde{\overline{W}} - D \to X - i(S)$ is an isomorphism. Thus, we obtain an abstract blowup square
\begin{equation}\label{eq:blowup-contraction-square}
\begin{tikzcd}[column sep=large]
D\simeq A \arrow[r,hook] \arrow[d,"p"'] &
\widetilde{\overline W} \arrow[d,"\rho"] \\
S \arrow[r,hook,"i"'] & X
\end{tikzcd}
\end{equation}
expressing that $X$ is obtained by contracting $D \subset \widetilde{\overline{W}}$. The following abstract lemma gives a mechanism to compute cohomology of any open subset of $X$ containing $i(S)$.

\begin{lemma}[Contracting an abelian variety]\label{lem:abelian-contraction}
Fix a noetherian base ring $k$.  Let $f:Y\to U$ be a proper birational morphism of normal $k$-schemes, and let
$u:\mathrm{Spec}(k) \to U$ be a $k$-point. Assume $f$
is an isomorphism away from $u$ and that
$D=f^{-1}(u)_{\mathrm{red}}$ is an effective Cartier divisor identified
with an abelian scheme $A/k$. If
$\mathcal N_D:=\mathcal O_D(-D)$ is ample, then the natural square
\[
\begin{tikzcd}[column sep=large]
\mathcal O_U \arrow[r] \arrow[d] &
Rf_*\mathcal O_Y \arrow[d] \\
u_*\mathcal O_{\mathrm{Spec}(k)} \arrow[r] &
Rf_* \mathcal{O}_A
\end{tikzcd}
\]
is a pullback square in $D(U)$. Consequently, we have a fibre sequence 
\[ \mathcal{O}_U \to Rf_* \mathcal{O}_Y \to  \tau^{>0} Rf_* \mathcal{O}_A.\]
\end{lemma}

\begin{proof}
We have $\mathcal O_U\simeq f_*\mathcal O_Y$, since $f$ is proper
birational and $U$ is normal, and
$k=H^0(A,\mathcal{O}_A)$  since $A$ is an abelian
scheme. It is therefore enough to show that the natural map
\[
R^{>0}f_*\mathcal O_Y\longrightarrow R^{>0}f_*\mathcal O_A
\]
is an isomorphism. Since $f$ is an isomorphism outside $u$, this follows
from the theorem on formal functions, the ampleness assumption on the
conormal bundle, and the fact that ample line bundles on abelian varieties
have no higher cohomology.
\end{proof}

\subsection{Cohomology after contraction}\label{subsec:cohomology-after-contraction}

In this subsection, we combine the calculations of \S\ref{subsec:blowup-cohomology} with the contraction from \S\ref{subsec:semiample-contraction} to compute the cohomology needed for local cohomology.

As both the modifications $b:\widetilde{\overline{W}} \to \overline{W}$ and $\rho:\widetilde{\overline{W}} \to X$ are happening away from $D_\infty \subset \overline{W}$, we can view $D_\infty$ as a Cartier divisor in any of those schemes. Pulling back
\eqref{eq:blowup-contraction-square} along $X-D_\infty\hookrightarrow X$ gives
\begin{equation}\label{eq:blowup-contraction-open-square}
\begin{tikzcd}[column sep=large]
D\simeq A \arrow[r,hook] \arrow[d,"p"'] &
\widetilde W=\widetilde{\overline W}-D_\infty
=\operatorname{Bl}_C(W) \arrow[d,"\rho_U"] \\
S \arrow[r,hook,"i"'] & U:=X-D_\infty.
\end{tikzcd}
\end{equation}

As $\widetilde{W} \to W$ is the blowup of $W$ along the codimension $2$ subscheme $C$ that is an ample divisor in $D_0 \subset W$,  the conormal bundle of the strict transform $D \subset \widetilde{W}$ of $D_0$ is given by twisting the conormal bundle of $D_0 \subset W$ up by $\mathcal{O}_A(C) = \mathcal{M}$, i.e., it is $\mathcal{L} \otimes \mathcal{M}$, which is $p$-ample.  Thus Lemma~\ref{lem:abelian-contraction} applies to
square~\eqref{eq:blowup-contraction-open-square}. Using this as well as Lemma~\ref{lem:blowup-tensor-power-cohomology}, we shall show:

\begin{lemma}[The cohomology of a contraction]
\label{lem:blowup-contraction-cohomology}
For every $m\in\mathbb Z$,
\[
H^1(U,\mathcal H^{\otimes m}|_U)\simeq
\begin{cases}
\displaystyle\bigoplus_{r>0}H^1(A,\mathcal L^{\otimes r}),&m=0,\\[6pt]
0,&m\neq0.
\end{cases}
\]
\end{lemma}

\begin{proof}
We shall compute the cohomology by applying
Lemma~\ref{lem:abelian-contraction} to the pullback square
\eqref{eq:blowup-contraction-open-square}. For this, we first compute the
pullback of $\mathcal H$ to the other three terms of the square.

The pullbacks of $\mathcal H$ to $S$ and $A\simeq D$ are trivial. Indeed,
using the commutativity of the square and the identity section for $A \to S$, it suffices to prove triviality of the pullback to $D \simeq A$. But $\rho_U^* \mathcal{H} = \mathcal{Q}|_{\widetilde{W}}^{\otimes e}$ by design, so the claim follows from the triviality
\eqref{eq:blowup-Q-on-D} of $\mathcal Q$ on $D$. 

On the other hand, on $\widetilde W$, the identity $\rho^*\mathcal H\simeq\mathcal Q^{\otimes e}$
and \eqref{eq:blowup-G} give
\[
\rho_U^*(\mathcal H^{\otimes m}|_U)
\simeq(\mathcal Q|_{\widetilde W})^{\otimes em}
\simeq\mathcal G^{\otimes em}.
\]
Note that the cohomology of this line bundle, as well as the restriction map from its
first cohomology to $H^1(D,\mathcal O_D)$ (coming from the previous paragraph), was computed in
Lemma~\ref{lem:blowup-tensor-power-cohomology}.

Because $S$ is affine and $p_*\mathcal O_A\simeq\mathcal O_S$, the cofiber of
$R\Gamma(S,\mathcal O_S)\to R\Gamma(A,\mathcal O_A)$ is
$\tau^{>0}R\Gamma(A,\mathcal O_A)$. Thus, by the projection formula and
Lemma~\ref{lem:abelian-contraction}, applying derived global sections to
square~\eqref{eq:blowup-contraction-open-square} gives a fibre sequence
\[
R\Gamma(U,\mathcal H^{\otimes m}|_U)\longrightarrow
R\Gamma(\widetilde W,\mathcal G^{\otimes em})\longrightarrow
\tau^{>0}R\Gamma(A,\mathcal O_A).
\]
The associated long exact sequence gives
\[
H^1(U,\mathcal H^{\otimes m}|_U)
\simeq\ker\!\left(
H^1(\widetilde W,\mathcal G^{\otimes em})
\longrightarrow H^1(A,\mathcal O_A)
\right),
\]
where the map is restriction to $D$. Since $e>0$, the computation and the
description of the restriction map in
Lemma~\ref{lem:blowup-tensor-power-cohomology} identify this kernel with
$\bigoplus_{r>0}H^1(A,\mathcal L^{\otimes r})$ when $m=0$, and with zero
when $m\neq0$.
\end{proof}

\subsection{Proof of the theorem}\label{subsec:proof-normal-five-blowups}

In this subsection, we put together \S\S\ref{subsec:blowup-cohomology}--\ref{subsec:cohomology-after-contraction} to prove Theorem~\ref{thm:normal-five-blowups-charp}.

\begin{proof}[Proof of Theorem~\ref{thm:normal-five-blowups-charp}]
Let
\[
T=\bigoplus_{n\geq0}H^0(X,\mathcal H^{\otimes n}),
\]
and let $J\subset T$ be the homogeneous ideal cutting out
$D_\infty\subset X$. The scheme $X$ is normal and integral, and
$\mathcal H$ is ample, so $T$ is a finitely generated normal $k$-domain. It
is flat over $k$. Since $\dim(X)=4$, we have $\dim(T)=5$. As in the proof of
Theorem~\ref{thm:normal-five-quotient}, we learn that $J$ is prime and $\dim(T/J)=4$, so $J$ has codimension $1$.

Lemmas~\ref{lem:cone-cohomology}
and~\ref{lem:blowup-contraction-cohomology} give
\[
H^2_J(T)\simeq
\bigoplus_{m\in\mathbb Z}H^1(U,\mathcal H^{\otimes m}|_U)
\simeq\bigoplus_{r>0}H^1(A,\mathcal L^{\otimes r}).
\]
This module is concentrated in degree $m=0$, and is therefore annihilated
by $T_+$. Moreover,
\[
T_0=H^0(X,\mathcal O_X)
=H^0(\widetilde{\overline W},\mathcal O_{\widetilde{\overline W}})
=\mathcal O(S).
\]
The argument at the end of the proof of Theorem~\ref{thm:normal-five-quotient}
therefore applies verbatim: the right-hand side has an infinite set of closed
points of $S$ as its minimal associated primes. Hence so does $H^2_J(T)$.
\end{proof}

\begin{remark}
The blow-up-then-contract operation used above to construct $X$ was also used in our previous work with Hochster \cite[Example~11.7 and Lemma~11.8]{BHM26}.
\end{remark}

\section{A normal local example via homogenization}

To pass from the affine example to a local one (i.e., from part (1) to part (2) in Theorem~\ref{main-theorem}) we use the affine cone construction:

\begin{lemma}[Global to local]
\label{lem:homogenization}
Let $A$ be a $d$-dimensional finitely generated normal $k$-domain, let
$J\subset A$ be a codimension $1$ prime ideal, and let $i\geq 0$. Suppose
that $H^i_J(A)$ has infinitely many minimal primes. Then there is a
$(d+1)$-dimensional normal local domain $A_{\mathrm{loc}}$, essentially of
finite type over $k$, and a codimension $1$ prime ideal
$J_{\mathrm{loc}}\subset A_{\mathrm{loc}}$ such that
\[
H^i_{J_{\mathrm{loc}}}(A_{\mathrm{loc}})
\]
has infinitely many minimal primes.
\end{lemma}

\begin{proof}
Write $X=\operatorname{Spec}(A)$ and choose a closed immersion
$X\hookrightarrow\mathbb A^N_k$. Let
$\overline X\subset\mathbb P^N_k$ be its projective closure. At the expense
of changing the initial closed immersion, we may assume that $\overline X$
is normal. Let $\mathcal L=\mathcal O_{\overline X}(1)$ and let
$s\in H^0(\overline X,\mathcal L)$ be the coordinate defining the hyperplane
at infinity. Thus $\mathcal L$ is ample and $X=\overline X_s$. Let
\[
B:=\bigoplus_{n\geq0}H^0(\overline X,\mathcal L^{\otimes n}),
\qquad
C:=\operatorname{Spec}(B),
\]
and let $v=V(B_+)\in C$ be the vertex. Since $\overline X$ is normal and
integral, its full section ring $B$ is a finitely generated normal domain of
dimension $d+1$. Moreover, $U:=D(s)\subset C$ is isomorphic to
$X\times\mathbb G_m$, because $(B_s)_0=\mathcal O(X)=A$ and the invertible
degree-$1$ element $s$ identifies $B_s$ with $A[s,s^{-1}]$.

Let $Z\subset X$ be the closed subscheme defined by $J$, and let
$Z_C\subset C$ be the scheme-theoretic closure of the inverse image of $Z$
in $U\simeq X\times\mathbb G_m$. Let $J_C\subset\mathcal O(C)$ be the ideal
of $Z_C$, and set
\[
A_{\mathrm{loc}}=\mathcal O(C)_v,
\qquad
J_{\mathrm{loc}}=(J_C)_v.
\]
The scheme $Z_C$ is integral of dimension $d$, and its homogeneous prime
ideal $J_C$ is contained in the vertex $v$. It follows that
$A_{\mathrm{loc}}$ is a $(d+1)$-dimensional normal local domain and
$J_{\mathrm{loc}}$ is a codimension $1$ prime ideal.

Let $M=H^i_J(A)$ and $N=H^i_{J_C}(B)$. As the projection
$f:U\simeq X\times\mathbb G_m\to X$ is flat, we have
$N|_U\simeq f^*M$. Hence
\[
\operatorname{Supp}N\cap U=f^{-1}(\operatorname{Supp}M).
\]
For each minimal point $\xi$ of $\operatorname{Supp}M$, the generic point
$\eta_\xi$ of the fibre
$f^{-1}(\xi)\simeq\mathbb G_m\times\operatorname{Spec}(\kappa(\xi))$ is
minimal in $\operatorname{Supp}N\cap U$. It is also minimal in
$\operatorname{Supp}N$: any point of $\operatorname{Supp}N$ specializing to
$\eta_\xi$ must lie in $U$, since the complement $C-U$ is closed. Since
supports are stable under specialization, the closure of $\eta_\xi$ in $C$
is contained in $\operatorname{Supp}N$. This closure contains the vertex
$v$, so the distinct points $\eta_\xi$ remain distinct minimal points after
localizing at $v$. Thus
$H^i_{J_{\mathrm{loc}}}(A_{\mathrm{loc}})$ has infinitely many minimal
primes.
\end{proof}

\begin{corollary}\label{cor:normal-local-six}
Let $k$ be a field. There is a normal local domain $R$ of dimension $6$,
essentially of finite type over $k$, and a codimension $1$ prime ideal
$I\subset R$ such that $H^2_I(R)$ has infinitely many minimal primes.
\end{corollary}

\begin{proof}
Apply Lemma~\ref{lem:homogenization} with $d=5$ and $i=2$ to the normal
domain and prime ideal constructed in
Theorem~\ref{thm:normal-five-quotient} if $\operatorname{char}(k)\neq2$, and
in Theorem~\ref{thm:normal-five-blowups-charp} if
$\operatorname{char}(k)=2$.
\end{proof}

\section{A \texorpdfstring{$5$}{5}-dimensional non-reduced example over
\texorpdfstring{$\mathbb Z$}{Z}}

In this section we give a completely explicit finite type algebra over
$\mathbb Z$ whose second local cohomology has infinitely many minimal
primes.

\begin{proposition}\label{prop:explicit-p-torsion}
Consider the finite type $\mathbb Z$-algebra
\[
R=\frac{\mathbb Z[x_0,x_1,x_2,u_0,u_1,u_2]}
{(x_0u_1-x_1u_0,\ x_0u_2-x_2u_0,\ x_1u_2-x_2u_1)^2}.
\]
Let $J=(x_0,x_1,x_2)$ and, for each prime number $p$, let
\[
\mathfrak m_p=(x_0,x_1,x_2,u_0,u_1,u_2,p).
\]
Then $R$ is non-reduced of dimension $5$, the ideal $J$ is a codimension $1$
prime ideal, and
\[
\operatorname{Supp}H^2_J(R)=\operatorname{Ass}H^2_J(R)
=\bigcup_p\{\mathfrak m_p\}.
\]
\end{proposition}

\begin{proof}
Let $Q=\mathbb Z[x_0,x_1,x_2,u_0,u_1,u_2]$, and let $I\subset Q$ denote the
ideal of $2\times2$ minors. The ring $Q$ has dimension $7$, the ideal $I$ is
a height $2$ prime, while
$(x_0,x_1,x_2)$ is a height $3$ prime containing it. Since
$\operatorname{Spec}(Q/I^2)=\operatorname{Spec}(Q/I)$ as topological spaces,
it follows that $\dim(R)=5$ and that $J$ is a codimension $1$ prime. Moreover,
$I^2$ is not radical, so $R$ is non-reduced.

It is clear that $H^2_J(R)[1/x_i]=0$ for each $i$. If we invert $u_i$, then
the ideal $(x_0,x_1,x_2)$ is generated up to radical by $x_i$, and thus
$H^2_J(R)[1/u_i]=0$ for each $i$. In particular, every prime ideal in the
support of $H^2_J(R)$ contains
\[
P:=(x_0,x_1,x_2,u_0,u_1,u_2).
\]
Thus, to prove the proposition, it suffices to show that
\begin{enumerate}
\item[(i)] $H^2_J(R)$ has $p$-torsion for every prime number $p$;
\item[(ii)] $H^2_J(R)\otimes_{\mathbb Z}\mathbb Q=0$.
\end{enumerate}
In fact, (i) implies that there is an associated prime of $H^2_J(R)$ that
contains $p$, and this prime must be $\mathfrak m_p$. Condition (ii) implies
that $P\notin\operatorname{Supp}H^2_J(R)$, so the proposition follows. Both
claims follow from Lemma~\ref{lem:explicit-calculation} below.
\end{proof}

For the rest of this section, set
\[
S=\mathbb Z[x_0,x_1,x_2],
\qquad
T=S[u_0,u_1,u_2],
\]
and
\[
I=(x_iu_j-x_ju_i\mid0\leq i<j\leq2),
\qquad
J=(x_0,x_1,x_2).
\]
Thus $R=T/I^2$. Give $T$, and hence $R$, the bigrading
\[
\deg x_i=(1,0),
\qquad
\deg u_i=(1,1).
\]
For each $r\in\mathbb N$, denote the $u$-degree $r$ pieces by
\[
A_r=\bigoplus_a(T/I^2)_{(a,r)},
\qquad
B_r=\bigoplus_a(T/I)_{(a,r)},
\qquad
N_r=\bigoplus_a(I/I^2)_{(a,r)}.
\]
Thus $A_r,B_r,N_r$ are graded $S$-modules with respect to the first grading,
which is the standard grading on $S$. For each $r\in\mathbb N$ there is an
exact sequence
\[
0\longrightarrow N_r\longrightarrow A_r\longrightarrow B_r
\longrightarrow0.
\]
We shall compute the degree-zero part, with respect to the first grading, of
the connecting map
\[
\delta_r:[H^1_J(B_r)]_0\longrightarrow[H^2_J(N_r)]_0.
\]
With this notation,
\[
[H^1_J(B_r)]_0=H^1_J(T/I)_{(0,r)},
\qquad
[H^2_J(N_r)]_0=H^2_J(I/I^2)_{(0,r)},
\]
and
\[
H^2_J(R)=\bigoplus_{r\in\mathbb N}H^2_J(A_r).
\]

\begin{lemma}\label{lem:first-lc}
There are isomorphisms $H^1_J(B_r)\simeq S/J^r$ of graded $S$-modules and
$H^2_J(B_r)=0$. In particular, $[H^1_J(B_r)]_0\simeq\mathbb Z$ for each
$r\geq1$, and an explicit \v{C}ech representative of the generator is
\[
\left(\frac{u_0^r}{x_0^r},
\frac{u_1^r}{x_1^r},
\frac{u_2^r}{x_2^r}\right).
\]
\end{lemma}

\begin{proof}
There is an identification $T/I\simeq S[Jt]$ sending $u_i$ to $x_it$, and
this isomorphism preserves the $u$-degree on $T/I$ and the $t$-degree on
$S[Jt]$. In particular, $B_r\simeq J^r$ as graded $S$-modules. The lemma
follows from the long exact sequence on local cohomology induced by
$0\to J^r\to S\to S/J^r\to0$.
\end{proof}

\begin{lemma}\label{lem:second-lc}
For each $r\geq1$, the coherent sheaf $\widetilde{N_r}$ on
$\operatorname{Proj}(S)=\mathbb P^2_{\mathbb Z}$ is isomorphic to
$\Omega^1_{\mathbb P^2_{\mathbb Z}}$, and
\[
H^2_J(N_r)\simeq[H^2_J(N_r)]_0
\simeq H^1(\mathbb P^2_{\mathbb Z},
\Omega^1_{\mathbb P^2_{\mathbb Z}})
\simeq\mathbb Z.
\]
\end{lemma}

\begin{proof}
We construct an explicit isomorphism
$\widetilde{N_r}\simeq\Omega^1_{\mathbb P^2_{\mathbb Z}}$. Set
\[
U_i=D_+(x_i),
\qquad
q_i=\frac{u_i}{x_i},
\qquad
\varepsilon_{ik}=\frac{x_iu_k-x_ku_i}{x_i^2}
\quad(k\neq i).
\]
The elements $q_i$ and $\varepsilon_{ik}$ have bidegree $(0,1)$. After
localizing at $x_i$, the ideal $I_{x_i}\subset T_{x_i}$ is generated by the
two elements $\varepsilon_{ik}$. It follows that
\[
(T/I^2)_{x_i}\simeq
S_{x_i}[q_i,\varepsilon_{ik}\mid k\neq i]/
(\varepsilon_{ik}\varepsilon_{i\ell}\mid k,\ell\neq i)
\]
and
\[
(I/I^2)_{x_i}\simeq
\bigoplus_{k\neq i}S_{x_i}[q_i]\cdot\varepsilon_{ik}.
\]
In particular,
\[
(N_r)_{x_i}\simeq
\bigoplus_{k\neq i}S_{x_i}\cdot q_i^{r-1}\varepsilon_{ik},
\]
and hence
\[
[(N_r)_{x_i}]_0\simeq
\bigoplus_{k\neq i}\mathcal O_{\mathbb P^2_{\mathbb Z}}(U_i)
\cdot q_i^{r-1}\varepsilon_{ik}.
\]
On $U_i$, define an isomorphism
\[
\Phi_i:[(N_r)_{x_i}]_0
\longrightarrow\Omega^1_{\mathbb P^2_{\mathbb Z}}(U_i)
\]
by
\[
\Phi_i(q_i^{r-1}\varepsilon_{ik})
=d\left(\frac{x_k}{x_i}\right).
\]
These local isomorphisms glue. For $r=1$, this follows directly from the
expression for $\varepsilon_{ik}$. For $r\geq2$, note that on
$U_i\cap U_j$ the difference $q_j-q_i$ belongs to $I/I^2$, and hence
$q_j^{r-1}\eta=q_i^{r-1}\eta$ for every $\eta\in I/I^2$. Thus multiplication
by $q_i^{r-1}$ gives a well-defined global identification of
$\widetilde{N_r}$ with
$\widetilde{N_1}\simeq\Omega^1_{\mathbb P^2_{\mathbb Z}}$.

By the module version of Lemma~\ref{lem:cone-cohomology},
\[
H^2_J(N_r)\simeq
\bigoplus_{n\in\mathbb Z}
H^1(\mathbb P^2_{\mathbb Z},
\Omega^1_{\mathbb P^2_{\mathbb Z}}(n)).
\]
The Euler sequence
\[
0\longrightarrow\Omega^1_{\mathbb P^2_{\mathbb Z}}
\longrightarrow\mathcal O_{\mathbb P^2_{\mathbb Z}}(-1)^{\oplus3}
\longrightarrow\mathcal O_{\mathbb P^2_{\mathbb Z}}
\longrightarrow0
\]
now gives the assertion.
\end{proof}

\begin{remark}\label{rmk:explicit-cech-generator}
We can identify an explicit generator of $[H^2_J(N_r)]_0$. Using the notation
of Lemma~\ref{lem:second-lc}, consider the \v{C}ech cocycle
\[
c_{ij}=q_i^{r-1}(q_j-q_i)
\qquad\text{on }U_i\cap U_j.
\]
Under the isomorphism
$\widetilde{N_r}\simeq\Omega^1_{\mathbb P^2_{\mathbb Z}}$, the cocycle
$(c_{ij})$ corresponds to the usual Euler generator of
$H^1(\mathbb P^2_{\mathbb Z},\Omega^1_{\mathbb P^2_{\mathbb Z}})$. Indeed,
\[
q_j-q_i
=\frac{u_j}{x_j}-\frac{u_i}{x_i}
=\frac{x_iu_j-x_ju_i}{x_ix_j}
=\frac{x_i}{x_j}\varepsilon_{ij},
\]
so
\[
c_{ij}=\frac{x_i}{x_j}q_i^{r-1}\varepsilon_{ij}.
\]
Applying $\Phi_i$ gives
\[
\Phi_i(c_{ij})
=\frac{x_i}{x_j}d\left(\frac{x_j}{x_i}\right)
=d\log\left(\frac{x_j}{x_i}\right),
\]
which is the standard Euler generator.
\end{remark}

\begin{lemma}\label{lem:connecting-map}
For each $r\geq 1$, the connecting map
\[
\delta_r:\mathbb Z\simeq[H^1_J(B_r)]_0
\longrightarrow[H^2_J(N_r)]_0\simeq\mathbb Z
\]
is multiplication by $r$.
\end{lemma}

\begin{proof}
Using the notation of Lemma~\ref{lem:second-lc}, the connecting map sends the
class represented by the \v{C}ech cocycle $(q_0^r,q_1^r,q_2^r)$ to the class
represented by
\[
(q_1^r-q_0^r,\ q_2^r-q_1^r,\ q_0^r-q_2^r)
\in[H^2_J(N_r)]_0.
\]
Since $q_j-q_i\in(I/I^2)_{x_ix_j}$, its square is zero in
$(T/I^2)_{x_ix_j}$, and therefore
\[
q_j^r-q_i^r
=(q_i+(q_j-q_i))^r-q_i^r
=r q_i^{r-1}(q_j-q_i).
\]
Comparison with the explicit generator in
Remark~\ref{rmk:explicit-cech-generator} completes the proof.
\end{proof}

\begin{lemma}\label{lem:explicit-calculation}
The module $H^2_J(R)$ contains $p$-torsion for every prime number $p$, while
$H^2_J(R)\otimes_{\mathbb Z}\mathbb Q=0$.
\end{lemma}

\begin{proof}
Consider the long exact sequence of local cohomology induced by
$0\to N_r\to A_r\to B_r\to0$. Lemmas~\ref{lem:first-lc},
\ref{lem:second-lc}, and~\ref{lem:connecting-map} give
\[
H^2_J(A_r)\simeq[H^2_J(A_r)]_0
\simeq\operatorname{coker}\left(
[H^1_J(B_r)]_0\xrightarrow{\delta_r}[H^2_J(N_r)]_0\right)
\simeq\mathbb Z/r\mathbb Z
\]
for each $r\geq1$. When $r=0$, we have
$H^2_J(A_r)=H^2_J(S)=0$. Consequently,
\[
H^2_J(R)\simeq
\bigoplus_{r\in\mathbb N}H^2_J(A_r)
\simeq\bigoplus_{r\geq1}\mathbb Z/r\mathbb Z,
\]
and the lemma follows.
\end{proof}

\begin{remark}
The $p$-torsion elements in local cohomology have been studied in \cite{Sin00}, where the first example of local cohomology with infinitely many associated primes was obtained. For a related purpose, the local cohomology of $T$ supported in $I$ was also studied in \cite{Sin00}.
\end{remark}

\bibliographystyle{alpha}
\bibliography{refs}

\end{document}